\documentclass[11pt,twoside]{amsart}
\usepackage{latexsym,amssymb,amsmath,curves}

\numberwithin{equation}{section}
\newtheorem{theorem}{Theorem}[section]
\newtheorem{lemma}[theorem]{Lemma}
\newtheorem{proposition}[theorem]{Proposition}
\newtheorem{corollary}[theorem]{Corollary}

\theoremstyle{definition}
\newtheorem{definition}[theorem]{Definition} 
 
\newtheorem{remark}[theorem]{Remark}

\begin{document}


\title[Codes and perfect matchings]{Codes parameterized by the edges of a bipartite graph with a perfect matching} 

\author[Manuel G. Sarabia]{Manuel Gonz\'alez Sarabia}
\address{
Instituto Polit\'ecnico Nacional, 
UPIITA, Av. IPN No. 2580,
Col. La Laguna Ticom\'an,
Gustavo A. Madero C.P. 07340,
 M\'exico, D.F. 
Departamento de Ciencias B\'asicas. 
}\email{mgonzalezsa@ipn.mx}

\author[Rafael H. Villarreal]{Rafael H. Villarreal}
\address{Departamento de Matem\'aticas,
Centro de Investigaci\'on y de Estudios Avanzados del IPN,
Apartado Postal 14--740, 07000, 
Ciudad de  M\'exico.}
\email{vila@math.cinvestav.mx}

\thanks{The first author is partially supported by COFAA--IPN and SNI--SEP. The second author is partially supported by SNI--SEP}

 
\begin{abstract}
In this paper we study the main characteristics of some evaluation codes pa\-ra\-me\-te\-rized by the edges of a bipartite graph with a perfect matching.
\end{abstract}

\maketitle
\section{Introduction}
This work aims to study certain classes of linear codes, known as parameterized codes (see Definition \ref{param1}). As our main goal is to relate these codes with bipartite graphs with a perfect matching (see Definitions \ref{graph1} and \ref{graph2}), the codes are parameterized by the edges of a graph $\mathcal{G}$. The procedure is as follows: given a graph $\mathcal{G}$, we define its toric set $\mathbb{X}$ parameterized by its edges (see Definition \ref{X}), and then we associate an evaluation code, $C_{\mathbb{X}}(d)$, to this set $\mathbb{X}$. Our primary purpose is the description of the main characteristics of these codes. This article is an interesting generalization of \cite{sarabia2}, where the authors study the case of an even cycle $\mathcal{G}=C_n$. Furthermore, this work generalizes the case where $\mathcal{G}=\mathcal{K}_{m,m}$, a specific complete bipartite graph. In both instances, $\mathcal{G}$ is a bipartite graph with a perfect matching.

As far as we know, the first approach to this topic is given in \cite{sarabia1}, where the authors study the codes $C_{\mathbb{X}}(d)$ when $\mathbb{X}$ is the toric set parameterized by the edges of a complete bipartite graph $\mathcal{K}_{m,n}$. The results they obtain come from the fact that this code is the tensor product (as linear spaces) of codes associated with the projective torus (see Definition \ref{torus}). The case of the codes $C_{\mathbb{X}}(d)$ when $\mathbb{X}$ is the projective torus $\mathbb{T}_{s-1}$ is studied in \cite{GRH}, although the dimension and the regularity index (see Definition \ref{regularity}) are found in \cite{duursma} because $\mathbb{X}$ is a complete intersection (see Definition \ref{ci}). However, the formula for the minimum distance is given until 2011, in \cite{ci-codes}. Actually, in 2018 and 2020, in \cite{GHWACC} and \cite{RGHWACC}, the authors found the generalized Hamming weights and the relative generalized Hamming weights, respectively, of the affine cartesian codes, which are introduced in \cite{hiram}. These weights are a generalization of the minimum distance, and since the codes arising from the projective torus are equal to some particular affine cartesian codes, their value is known when $\mathbb{X}=\mathbb{T}_{s-1}$. Furthermore, the study of the generalized Hamming weights in the case of the cycle $C_4$ and some complete bipartite graphs of the form $\mathcal{K}_{2,m}$ is given in \cite{GHW2014}.

The only parameter known for any simple graph, connected or not, is the length of the code. In 2015, \cite[Theorem 3.2]{vaz}, the authors found an explicit formula for the cardinality of the set $\mathbb{X}$, which is the length of $C_{\mathbb{X}}(d)$. Also, in the same article, they found the regularity index when $\mathcal{G}$ is an even cycle, \cite[Theorem 6.2]{vaz}, and an upper bound for the case of bipartite graphs with subgraphs isomorphic to even cycles that have disjoint edge sets, \cite[Theorem 6.3]{vaz}. This upper bound is attained if the graph is connected, \cite[Corollary 6.5]{vaz}. Moreover, the case when $\mathcal{G}$ is an odd cycle is completely solved because its toric set is the projective torus.

If $\mathcal{G}=\mathcal{K}_n$ is a complete graph, its regularity index is shown in \cite[Remark 3]{sarabia3}. Moreover, some bounds for the minimum distance of these codes are given in \cite[Corollaries 8 and 9]{sarabia3}. Also, the regularity index when $\mathcal{G}$ is a complete multipartite graph is computed in \cite[Theorem 4.3]{neves}. Finally, in \cite{sarabia6}, there are some general bounds for the main parameters of the code $C_{\mathbb{X}}(d)$ when $\mathbb{X}$ is the toric set parameterized by the edges of any simple graph $\mathcal{G}$.

The contents of this paper are as follows. In Section \ref{preliminaries}, we introduce the main concepts about graphs and linear codes that will be useful to the development of the article. We define the code $C_{\mathbb{X}}(d)$ when $\mathbb{X}$ is the toric set parameterized by the edges of a bipartite graph with a perfect matching, which is the fundamental structure of this work. In Section \ref{bounds}, we give, in Theorem \ref{theorem1}, some bounds for the dimension and the minimum distance of $C_{\mathbb{X}}(d)$, and also for the regularity index of the vanishing ideal $I_{\mathbb{X}}$. In Section \ref{vanishing}, we define the set $\mathbb{Y}$, which plays a significant role studying the dimension of $C_{\mathbb{X}}(d)$. We prove, in Theorem \ref{theorem2}, that $\mathbb{Y}$ is a complete intersection, and find a set of generators for its vanishing ideal $I_{\mathbb{Y}}$. Moreover, we give a formula relating $I_{\mathbb{X}}$ and $I_{\mathbb{Y}}$ in Proposition \ref{prop2}.
In Section \ref{DimReg}, we find, in Theorem \ref{TeoF}, a formula for the dimension of the code $C_{\mathbb{X}}(d)$ in terms of the dimension of the codes asso\-cia\-ted with the projective torus, and the Hilbert function of $I_{\mathbb{Y}}/I_{\mathbb{X}}$. It allows us to give a tight lower bound for the regularity index of the vanishing ideal $I_{\mathbb{X}}$, which is attained in the cases of even cycles and complete bipartite graphs of the form $\mathcal{K}_{m,m}$. This bound is also attained in the case of graphs such that all of their connected components are even cycles (Corollary \ref{corollary2}). Finally, in Section \ref{example}, we give an example of a graph with two connected components, each of them a square, and describe the main characteristics of the code $C_{\mathbb{X}}(d)$ that were obtained in this work.

\section{Preliminaries} \label{preliminaries}
Let $\mathcal{G}=(\mathcal{V},\mathcal{E})$ be a simple graph with vertex set $\mathcal{V}=\{v_1,\ldots,v_n\}$ and edge set $\mathcal{E}=\{e_1,\ldots,e_s\}$.
\begin{definition} \label{graph1}
A graph $\mathcal{G}=(\mathcal{V},\mathcal{E})$ is called bipartite if there is a partition of $\mathcal{V}$ into two disjoint subsets $\mathcal{V}=\mathcal{U} \cup \mathcal{W}$, such that every edge $e \in \mathcal{E}$ joins a vertex in $\mathcal{U}$ to a vertex in $\mathcal{W}$.
\end{definition}

\begin{definition} \label{graph2}
A matching $\mathcal{M}$ in $\mathcal{G}$ is a subset of the edge set $\mathcal{E}$ such that for every $e, e' \in \mathcal{M}$ there is no vertex $v \in \mathcal{V}$ such that $e$ and $e'$ are both incidents on $v$. The matching $\mathcal{M}$ is called perfect if, for every $v \in \mathcal{V}$, there is $e \in \mathcal{M}$ which is incident on $v$.
\end{definition}

It is immediate that if $\mathcal{G}$ has a perfect matching, then $|\mathcal{V}|$ is an even number. If $\mathcal{G}$ is bipartite, then $|\mathcal{U}|=|\mathcal{W}|$. From now on we assume that $\mathcal{G}$ is bipartite, $n=2k$, and $\mathcal{M}=\{e_1,e_2,\ldots,e_k\}$ is a perfect matching. Without loss of generality we take $e_i=\{v_{2i-1},v_{2i}\}$ for all $i=1,\ldots,k$. Therefore we set $\mathcal{U}=\{v_1,v_3,\ldots,v_{2k-1}\}$ and $\mathcal{W}=\{v_2,v_4,\ldots,v_{2k}\}$. Also, from now on, we denote the degree of each vertex $v_i$ as $n_i$ for all $i=1,\ldots,n$. For this kind of graphs we notice that $n_1+n_3+\cdots+n_{2k-1}=s$.

Let $K=\mathbb{F}_q$ be a finite field with $q$ elements. The set of non--zero elements of $K$ is denoted by $K^*$, and $|\mathbb{X}|$ denotes the cardinality of any set $\mathbb{X}$.
\begin{definition} \label{torus}
The projective torus of dimension $s-1$, which is a group under componentwise multiplication, is given by
$$
\mathbb{T}_{s-1}=\{[t_1,\ldots,t_s] \in \mathbb{P}^{s-1}: (t_1,\ldots,t_s) \in (K^*)^s \},
$$
where the projective space $ \mathbb{P}^{s-1}$ is the quotient space $(K^s \setminus \{\bf{0}\}) / \sim$, where for any $\vec{x}_1$, $\vec{x}_2 \in K^s \setminus \{\bf{0}\}$, $\vec{x}_1 \sim \vec{x}_2$ if and only if there is $\lambda \in K^*$ such that $\vec{x}_1 =\lambda \vec{x}_2$.
\end{definition}

Furthermore, we need to introduce some basic facts about linear codes and how we define the linear codes parameterized by the edges of a graph $\mathcal{G}$.
We consider that $\mathcal{G}$ has no isolated vertices and it is not necessarily connected. Also we assume that $\mathcal{V}=\{v_1,\ldots,v_n\}$ is the vertex set, and ${\mathcal{E}}=\{e_1,\ldots,e_s\}$ is the edge set of $\mathcal{G}$. For each edge $e_i=\{v_j,v_k\}$, where $v_j,v_k \in \mathcal{V}$, let $(t_1,\ldots,t_n)^{e_i}=t_jt_k$ for $(t_1,\ldots,t_n) \in (K^*)^n$.
\begin{definition} \label{X}
The toric set $\mathbb{X}$ parameterized by the edges of the graph $\mathcal{G}$ is the following subset of the projective torus $\mathbb{T}_{s-1}$:
\begin{equation} \label{toricset1}
\mathbb{X}=\{[(t_1,\ldots,t_n)^{e_1},\ldots,(t_1,\ldots t_n)^{e_s}] \in \mathbb{P}^{s-1}: t_i \in K^*\}.
\end{equation}
\end{definition}

Equation (\ref{toricset1}) works for any simple graph. However, when we work with the case where $\mathcal{G}$ is a bipartite graph with a perfect matching and with $m$ connected components, there is no loss of generality  if we assume that the toric set parameterized by its edges is given by
$$
\mathbb{X}=\{[t_1t_2,\ldots,t_1t_{2i_1},t_3t_4,\ldots,t_3t_{2i_3},\ldots,t_{2k-1}t_{2k},  \ldots,  t_{2k-1}t_{2i_{2k-1}}] \
  \in \mathbb{P}^{s-1}: t_i \in K^*\},
$$
where the first $n_1$ entries are the edges incident on $v_1$ (starting with $t_1t_2$), the second block has $n_3$ entries which are the edges incident on $v_3$ (starting with $t_3t_4$), and so on until the last block of entries which are the $n_{2k-1} $ edges incident with $v_{2k-1}$ (starting with $t_{2k-1}t_{2k}$). Each block of entries  starts with the elements of the perfect matching $\mathcal{M}$.

\begin{definition}
A linear code $C$ is a subspace of $K^l$, where $l$ is a positive integer. This integer $l$ is known as its length. Its dimension as a linear space over $K$ is called its dimension, and it is denoted by $\dim_K C$. Finally, the minimum distance of a code $C$, $\delta_C$, is defined as follows:
$$
\delta_C=\min \{w_H(v) : v \in C, \,\, v \neq 0\},
$$
where $w_H(v)$ is the Hamming weight of $v$, that is, the number of non--zero entries of $v$. These three numbers (length, dimension, and minimum distance) are called the main parameters of a code $C$, and they are related by the Singleton bound:
$$
\delta_C \leq l-\dim_K C+1.
$$
\end{definition}

Moreover, let $S=K[X_1,\ldots,X_s]=\oplus_{d \geq 0} S_d$ be a polynomial ring and $\mathbb{X}=\{P_1,\ldots,P_{|\mathbb{X}|}\}$ be the toric set parameterized by the edges of the graph $\mathcal{G}$.
\begin{definition} \label{param1}
The code of order $d$ parameterized by the edges of the graph $\mathcal{G}$, which is denoted by $C_{\mathbb{X}}(d)$, is the following subspace of $K^{|\mathbb{X}|}$:
$$
C_{\mathbb{X}}(d)=\left\{\left( \frac{f(P_1)}{X_1^d(P_1)},\ldots,\frac{f(P_{|{\mathbb{X}}|})}{X_1^d(P_{|\mathbb{X}|})}\right) : f \in S_d\right\},
$$ 

that is, $C_{\mathbb{X}}(d)$ is the image of the following surjective linear transformation:
\begin{align*}
S_d \longrightarrow K^{|\mathbb{X}|}, \hspace{1.4cm}\\
f \mapsto \left( \frac{f(P_1)}{X_1^d(P_1)},\ldots,\frac{f(P_{|{\mathbb{X}}|})}{X_1^d(P_{|\mathbb{X}|})}\right).
\end{align*}

\end{definition}
\begin{definition}
The ideal of $S$ that is spanned by the homogeneous polynomials that vanish on $\mathbb{X}$ is called the vanishing ideal of $\mathbb{X}$, and it is denoted by $I_{\mathbb{X}}$. It is a graded ideal, $I_{\mathbb{X}}=\oplus_{d \geq 0} I_{\mathbb{X}}(d)$, and its main characteristics can be seen in \cite{algcodes}.
\end{definition}

In the case of the code $C_{\mathbb{X}}(d)$, its length is $|\mathbb{X}|$, its dimension is given by the Hilbert function $\dim_K C_{\mathbb{X}}(d)=H_{\mathbb{X}}(d)=\dim_K (S_d/I_{\mathbb{X}}(d))$, and its minimum distance is denoted by $\delta_{\mathbb{X}}(d)$. It is known that the Hilbert function is a strictly increasing function until it stabilizes.

\begin{definition} \label{regularity}
The regularity index of $S/I_{\mathbb{X}}$ is the least integer $d$ such that $H_{\mathbb{X}}(d)=|\mathbb{X}|$. It is denoted by ${\rm{reg}} \, (S/I_{\mathbb{X}})$. Actually, $H_{\mathbb{X}}(d)=|\mathbb{X}|$ for all $d \geq {\rm{reg}} \, (S/I_{\mathbb{X}})$. For these cases, $C_{\mathbb{X}}(d)=K^{|\mathbb{X}|}$, and then $\delta_{\mathbb{X}}(d)=1$. Therefore the only interesting codes $C_{\mathbb{X}}(d)$ are those for which $d<{\rm{reg}} \, (S/I_{\mathbb{X}})$.
\end{definition}

Finally, we need to introduce the concept of a complete intersection.

\begin{definition} \label{ci}
A set of points $\mathbb{X} \subseteq \mathbb{P}^{s-1}$ is called a complete intersection if its vanishing ideal $I_{\mathbb{X}}$ is generated by a regular sequence of $s-1$ elements, that is,
$
I_{\mathbb{X}}=(f_1,\ldots,f_{s-1}),
$
such that $f_1$ is not a zero divisor of $S$, and $f_i$ is not a zero divisor of $S/(f_1,\ldots,f_{i-1})$ for $i=2,\ldots,s-1$.
\end{definition}

\section{Some bounds} \label{bounds}
In the following theorem, we give the length of the code $C_{\mathbb{X}}(d)$ parameterized by the edges of a bipartite graph with $m$ co\-nnec\-ted components and with a perfect matching, and also we give some bounds for the regularity index of $S/I_{\mathbb{X}}$, the dimension, and the minimum distance. It is worth mentioning that the bounds for the minimum distance of $C_{\mathbb{X}}(d)$ depend on the value of the minimum distance when the toric set is the projective torus, which was computed in \cite[Theorem 3.5]{ci-codes}.

\begin{theorem} \label{theorem1}
If $\mathcal{G}$ is a bipartite graph with a perfect matching, with $m$ co\-nnec\-ted components, and $\mathbb{X}$ is the toric set parameterized by its edges, then:

\begin{enumerate}
\item The length of the code $C_{\mathbb{X}}(d)$ is given by: \label{1}
$$
|\mathbb{X}|=(q-1)^{n-m-1}.
$$
\item The regularity index of $S/I_{\mathbb{X}}$ is bounded by: \label{2}
$$
\left\lceil \frac{(q-2)(n-1)}{2 (q-1)^{m}}\right\rceil \leq {\rm{reg}} \, (S/I_{{\mathbb{X}}}) \leq (q-2)(k-1)+(q-1)^{k-m}-1,
$$
where $n=2k$.
\item The dimension of the code $C_{\mathbb{X}}(d)$ is bounded by: \label{3}
$$
\dim_K(C_{\mathbb{X}}(d)) \geq \sum_{j=0}^{\lfloor d/(q-1) \rfloor} (-1)^j \binom{k-1}{j} \binom{k-1+d-j(q-1)}{k-1}.
$$
\item The minimum distance of the code $C_{\mathbb{X}}(d)$ is bounded by: \label{4}
$$
\left\lceil  \frac{\delta_{\mathbb{T}_{n-1}}(2d)}{(q-1)^m}  \right\rceil \leq \delta_{\mathbb{X}}(d) \leq (q-1)^{k-m} \delta_{\mathbb{T}_{k-1}}(d).
$$
\end{enumerate}

\begin{proof}
Assertion (\ref{1}) follows directly from \cite[Theorem 3.2]{vaz}. Moreover, the lower bounds for the regularity index and the minimum distance given in (\ref{2}) and (\ref{4}) follow from \cite[Theorem 2]{sarabia6}, because $\gamma$, the number on non--bipartite components of $\mathcal{G}$, is equal to zero.

Furthermore, let $\mathcal{G}'=(\mathcal{V},\mathcal{E}')$ be the subgraph of $\mathcal{G}$ with the same vertex set, but with $\mathcal{E}'=\mathcal{M}$. Let $\mathbb{X}$ and $\mathbb{X}'$ be the toric sets associated with the edges of $\mathcal{G}$ and $\mathcal{G}'$, respectively. Thus
\begin{equation}
\mathbb{X}'=\{[t_1t_2,t_3t_4, \cdots,t_{2k-1}t_{2k}] \in \mathbb{P}^{k-1}: t_i \in K^*\}, \label{Perfect}
\end{equation}
and therefore $\mathbb{X}'=\mathbb{T}_{k-1}$. Then $|\mathbb{X}'|=(q-1)^{k-1}$, and we notice that
$$
|\mathbb{X}|=(q-1)^{n-m-1}=|\mathbb{X}'| (q-1)^{k-m}.
$$

Inequality (\ref{3}) and the upper bounds for the regularity index and the mi\-ni\-mum distance given in (\ref{2}) y (\ref{4}) follow easily from \cite[Theorem 3]{sarabia6}.
\end{proof}
\end{theorem}

\section{Vanishing ideals} \label{vanishing}
We continue using the notation of the last sections. Let $n_i=\deg v_i$ for $i=1,\ldots,n$. Of course, $n_1+n_3+\cdots +n_{2k-1}=s$, the number of edges of the graph $\mathcal{G}$. Moreover,
let $S=K[X_1,\ldots,X_s]$, as in Section \ref{preliminaries}, and $R=K[Y_1,Y_3,\ldots,Y_{2k-1}]$ be two polynomial rings. 
From now on, let $\mathbb{Y}$ be the following subset of the projective space $\mathbb{P}^{s-1}$:
\begin{equation} \label{Intersection}
\mathbb{Y}=\{[t_1,\ldots,t_{1},t_3,\ldots,t_{3},\ldots,t_{2k-1},\ldots,t_{2k-1}]  \in \mathbb{P}^{s-1}: t_i \in \mathbb{K}^{*}\},
\end{equation}
where $t_i$ appears $n_i$ times, for all $i=1,3,\ldots,2k-1$. Clearly, $\mathbb{Y} \subseteq \mathbb{X}$. Actually, we get the following result:

\begin{theorem} \label{theorem2}
$\mathbb{Y}$ is a complete intersection. In fact, $I_{\mathbb{Y}}$ is spanned by the union of the following sets:
\begin{enumerate}
\item $\mathbb{W}_0=\{X_{n_1+1}^{q-1}-X_1^{q-1}, X_{n_1+n_3+1}^{q-1}-X_1^{q-1},\ldots,X_{n_1+n_3+\cdots+n_{2k-3}+1}^{q-1}-X_1^{q-1}\}$.
\\ \\
\item $\mathbb{W}_1=\left\{ 
\begin{array}{ccc}
\emptyset & {\mbox{if}} & n_1=1, \\
&& \\ 
\{X_i-X_1\}_{i=2}^{n_1} &  {\mbox{if}} & n_1 \geq 2.
\end{array} \right.
$
\\ \\
\item $\mathbb{W}_3=\left\{ 
\begin{array}{ccc}
\emptyset & {\mbox{if}} & n_3=1, \\
&& \\
\{X_i-X_3\}_{i=n_1+2}^{n_1+n_3} &  {\mbox{if}} & n_3 \geq 2.
\end{array} \right.
$
\\  \\ \vdots \hspace{1cm} \vdots \hspace{1cm} \vdots \hspace{1cm} \vdots \hspace{1cm} \vdots
\\ \\
\item $\mathbb{W}_{2k-1}=\left\{ 
\begin{array}{ccc}
\emptyset & {\mbox{if}} & n_{2k-1}=1, \\
&& \\
\{X_i-X_{2k-1}\}_{i=n_1+n_3+\cdots +n_{2k-3}+2}^{s} &  {\mbox{if}} & n_{2k-1} \geq 2.
\end{array} \right.
$
\end{enumerate}
\end{theorem}

\begin{proof}
We want to show that $(\mathbb{W})=I_{\mathbb{Y}}$, where $\mathbb{W}:=\mathbb{W}_0 \cup \mathbb{W}_1\cup  \mathbb{W}_3 \cup \cdots \cup \mathbb{W}_{2k-1}$.
At first, we notice that the number of polynomials in $\mathbb{W}$ is given by
\begin{align*}
|\mathbb{W}| & =|\mathbb{W}_0|+|\mathbb{W}_1| + |\mathbb{W}_3|+ \cdots + |\mathbb{W}_{2k-1}| \\
& = (k-1)+(n_1-1)+(n_3-1)+\cdots +(n_{2k-1}-1) \\
& = k-1+s-k=s-1,
\end{align*}
because $n_1+n_3+\cdots + n_{2k-1}=s$.
On the other hand, it is easy to see that $\mathbb{W} \subseteq I_{\mathbb{Y}}$. 
Furthermore, let $f \in I_{\mathbb{Y}}$. We use the lexicographic ordering $X_s>X_{s-1}> \cdots > X_1$. By the division algorithm, $f$ can be written as:
$$
f=\sum_{i=2}^{n_1} f_i(X_i-X_1) +\sum_{i=n_1+2}^{n_1+n_3}  g_i(X_i-X_3)+ \cdots +  \sum_{i=n_1+n_3+ \cdots +n_{2k-3}+2}^s 
 h_i(X_i-X_{2k-1})+r,
$$
where $f_i,g_i \ldots, h_i, r \in S$, and $r=0$ (in this case $f \in (\mathbb{W})$, and $I_{\mathbb{Y}} \subseteq (\mathbb{W})$) or $r$ is a $K$--linear combination of monomials, none of which is divisible by any of 
$$
X_2,\ldots,X_{n_1},X_{n_1+2},\ldots,X_{n_1+n_3},\ldots,X_{n_1+n_3+\cdots +n_{2k-3}+2},\ldots,X_s.
$$

Therefore, $$r(X_1,\ldots,X_s)=r(X_1,X_{n_1+1},X_{n_1+n_3+1},\ldots,X_{n_1+n_3+\cdots +n_{2k-3}+1}).$$ 

Also, we observe that the projective torus $\mathbb{T}_{k-1}$ can be written as
$$
\mathbb{T}_{k-1}=\{[t_1,t_{n_1+1},t_{n_1+n_3+1},\ldots,t_{n_1+n_3+\cdots +n_{2k-3}+1}]: t_i \in K^*\},
$$
and thus (\cite[Theorem 1]{GRH})
\begin{equation} \label{torus2}
I_{\mathbb{T}_{k-1}}=(X_{n_1+1}^{q-1}-X_1^{q-1},  X_{n_1+n_3+1}^{q-1}  -X_1^{q-1},
\ldots,X_{n_1+n_3+\cdots+n_{2k-3}+1}^{q-1}-X_1^{q-1}).
\end{equation}

Let $P:=[t_1,\ldots,t_{1},t_3,\ldots,t_{3},\ldots,t_{2k-1},\ldots,t_{2k-1}] \in \mathbb{Y}$. Then
$$
0=f(P)=r(P)=r(t_1,t_{n_1+1},t_{n_1+n_3+1},\ldots,t_{n_1+n_2+\cdots +n_{2k-3}+1}),
$$
for all $t_i \in K^*$.
Therefore $r \in I_{\mathbb{T}_{k-1}}$, and by Equation (\ref{torus2}), $r \in (\mathbb{W}_0)$, and thus $f \in (\mathbb{W})$. That is, $I_{\mathbb{Y}}=(\mathbb{W})$, and the claim follows
\end{proof}

\noindent
On the other hand,
let $\theta$ be the map
\begin{align*}
& \hspace{4.5cm} \theta: S \rightarrow R, \\
& f(X_1,\ldots,X_s) \mapsto f(Y_1,\ldots,Y_1,Y_3,\ldots,Y_3,\ldots Y_{2k-1},\ldots,Y_{2k-1}),
\end{align*}
where $Y_i$ appears $n_i$ times, for all $i=1,3,\ldots,2k-1$.
We notice that $\theta(X_i)=Y_1$ for all $i=1,\ldots,n_1$, $\theta(X_i)=Y_3$ for all $i=n_1+1,\ldots,n_1+n_3$, and so on until $\theta(X_i)=Y_{2k-1}$ for all $i=n_1+n_3+ \cdots + n_{2k-3}+1,\ldots,s$.  
Moreover, the following result relates the vanishing ideals of $\mathbb{X}$ and $\mathbb{T}_{k-1}$.

\begin{proposition} \label{prop1}
With the notation introduced above, $\theta$ is a ring epimorphism and
$$
\theta(I_{\mathbb{X}})=I_{\mathbb{T}_{k-1}}.
$$
\end{proposition}

\begin{proof}
The fact that $\theta$ is a ring epimorphism follows directly from the de\-fi\-ni\-tions. Let $f \in I_{\mathbb{X}}$ and $Q=[t_1,t_3,\ldots,t_{2k-1}] \in \mathbb{T}_{k-1}$. Also let $P=[t_1,\ldots,t_1,t_3,\ldots,t_3,\ldots,t_{2k-1},\ldots,t_{2k-1}] \in \mathbb{Y} \subseteq \mathbb{X}$. Therefore
$$
\theta(f)(Q)=f(P)=0.
$$
Thus $\theta(f) \in I_{\mathbb{T}_{k-1}}$ and then $\theta(I_{\mathbb{X}}) \subseteq I_{\mathbb{T}_{k-1}}$.
On the other hand, as $\mathbb{X} \subseteq \mathbb{T}_{s-1}$, we get that $I_{\mathbb{T}_{s-1}} \subseteq I_{\mathbb{X}}$. Furthermore, $I_{\mathbb{T}_{k-1}}$ can be written as
$$
I_{\mathbb{T}_{k-1}}=(Y_3^{q-1}-Y_1^{q-1},\ldots,Y_{2k-1}^{q-1}-Y_1^{q-1}).
$$

But
$$
Y_3^{q-1}-Y_1^{q-1} =\theta(X_{n_1+1}^{q-1}-X_1^{q-1}),\ldots, Y_{2k-1}^{q-1}-Y_1^{q-1} 
=\theta(X_{n_1+n_3+\cdots+n_{2k-3}+1}^{q-1}-X_1^{q-1}).
$$

Moreover,
$$
\{X_{n_1+1}^{q-1}-X_1^{q-1},\ldots, X_{n_1+n_3+\cdots+n_{2k-3}+1}^{q-1}-X_1^{q-1}\} \subseteq I_{\mathbb{T}_s-1} \subseteq I_{\mathbb{X}}.
$$

Therefore $I_{\mathbb{T}_{k-1}} \subseteq \theta (I_{\mathbb{X}})$, and the claim follows.
\end{proof}

Since $\mathbb{Y} \subseteq \mathbb{X}$ we obtain that $I_{\mathbb{X}} \subseteq I_{\mathbb{Y}}$. The following result relates the vanishing ideals of $\mathbb{X}$, $\mathbb{Y}$, and the map $\theta$.

\begin{proposition} \label{prop2}
The vanishing ideal of $\mathbb{Y}$ is given by
$$
I_{\mathbb{Y}}=I_{\mathbb{X}}+ \ker \theta.
$$
\end{proposition}

\begin{proof}
We notice that $\mathbb{W}_0 \subseteq I_{\mathbb{T}_{s-1}} \subseteq I_{\mathbb{X}}$ and $\mathbb{W}_{2i-1} \subseteq \ker \theta$ for all $i=1,\ldots,k$. Therefore, by using Theorem \ref{theorem1}, we conclude that
\begin{equation} \label{ideal}
I_{\mathbb{Y}} \subseteq I_{\mathbb{X}}+ \ker \theta.
\end{equation} 

Furthermore, let $f \in I_{\mathbb{X}}$, $g \in \ker \theta$, and $$P:=[t_1,\ldots,t_{1},t_3,\ldots,t_{3},\ldots,t_{2k-1},\ldots,t_{2k-1}] \in \mathbb{Y}.$$ 

Thus, because $P \in \mathbb{Y} \subseteq \mathbb{X}$, $f(P)=0$. Moreover, because $g \in \ker \theta$,
$$
0=\theta(g)(t_1,t_3,\ldots,t_{2k-1})=g(P),
$$
and then $(f+g)(P)=0$. That is,
\begin{equation} \label{ideal2}
I_{\mathbb{X}}+\ker \theta \subseteq I_{\mathbb{Y}}.
\end{equation}

The claim follows from (\ref{ideal}) and (\ref{ideal2}).
\end{proof}

\section{Dimension and the regularity index} \label{DimReg}

Now, our goal is to compute the dimension of the code $C_{\mathbb{X}}(d)$ parameterized by the edges of a bipartite graph $\mathcal{G}$ with a perfect matching. To do this, we need the following lemma.

\begin{lemma}
Let $\psi$ the following map:
\begin{align*}
& \hspace{0.2cm} \psi: S/I_{\mathbb{X}} \rightarrow R/I_{\mathbb{T}_{k-1}}, \\
& f + I_{\mathbb{X}} \rightarrow \theta(f)+I_{\mathbb{T}_{k-1}}.
\end{align*}

Therefore, $\psi$ is a ring epimorphism, and $\ker \psi=I_{\mathbb{Y}}/I_{\mathbb{X}}$.
\end{lemma}

\begin{proof}
At first, we notice that  the map $\psi$ is well--defined because if $f+I_{\mathbb{X}}=g+I_{\mathbb{X}}$ then $f-g \in I_{\mathbb{X}}$. Thus, because of Proposition \ref{prop1}, $\theta(f)-\theta(g)=\theta(f-g) \in \theta(I_{\mathbb{X}})=I_{\mathbb{T}_{k-1}}$, and therefore $\theta(f)+I_{\mathbb{T}_{k-1}}=\theta(g)+I_{\mathbb{T}_{k-1}}$.
Also,
that $\psi$ is a ring epimorphism follows immediately from the fact that $\theta$ is a ring epimorphism.

Let $f +I_{\mathbb{X}} \in \ker \psi$. Then $\theta(f) \in I_{\mathbb{T}_{k-1}}=\theta(I_{\mathbb{X}})$. Thus, there exists $g \in I_{\mathbb{X}}$ such that $\theta(f)=\theta(g)$. That is, $\theta(f-g)=\theta(f)-\theta(g)=0$, and it implies that $f-g \in \ker \theta$. Therefore, there exists $h \in \ker \theta$ such that $f-g=h$, that is, $f=g+h$. Then $f \in I_{\mathbb{X}}+\ker \theta=I_{\mathbb{Y}}$, and we conclude that
\begin{equation} \label{inclusion1}
\ker \psi \subseteq I_{\mathbb{Y}}/I_{\mathbb{X}}.
\end{equation}

On the other hand, let $f+I_{\mathbb{X}} \in I_{\mathbb{Y}}/I_{\mathbb{X}}$. As $f \in I_{\mathbb{Y}}=I_{\mathbb{X}}+\ker \theta$, we get that $f=f_1+f_2$ for some $f_1 \in I_{\mathbb{X}}$, $f_2 \in \ker \theta$. Thus $\theta(f)=\theta(f_1)+\theta(f_2)=\theta(f_1) \in \theta(I_{\mathbb{X}})=I_{\mathbb{T}_{k-1}}$. Therefore $\psi(f+I_{\mathbb{X}})=\theta(f)+I_{\mathbb{T}_{k-1}}=I_{\mathbb{T}_{k-1}}$, and then $f+I_{\mathbb{X}} \in \ker \psi$. It proves that
\begin{equation} \label{inclusion2}
I_{\mathbb{Y}}/I_{\mathbb{X}} \subseteq \ker \psi,
\end{equation}
and the result follows from (\ref{inclusion1}) and (\ref{inclusion2}).
\end{proof}

From now on, we set $H_{\psi}(d):=\dim_K(I_{\mathbb{Y}}(d)/I_{\mathbb{X}}(d))$ for all $d \geq 0$. The main result of this section gives the dimension of $C_{\mathbb{X}}(d)$ in terms of $H_{\psi}(d)$ and $H_{\mathbb{T}_{k-1}}(d)$, where we know that (\cite[Lemma 1]{GRH} and \cite[Corollary 2.2]{ci-codes})
$$
H_{\mathbb{T}_{k-1}}(d)=\sum_{j=0}^{\lfloor \frac{d}{q-1} \rfloor} (-1)^j \binom{k-1}{j} \binom{k-1+d-j(q-1)}{k-1}.
$$

\begin{theorem} \label{dimension} \label{TeoF}
The dimension of the code $C_{\mathbb{X}}(d)$ parameterized by the edges of a bipartite graph $\mathcal{G}$ with a perfect matching is given by
$$
\dim_K C_{\mathbb{X}}(d)=H_{\mathbb{X}}(d)=H_{\psi}(d)+H_{\mathbb{T}_{k-1}}(d),
$$
for all $d \in \mathbb{Z}, d \geq 0$.
\end{theorem}

\begin{proof}
Let $\psi_d$ be the following linear map
\begin{align*}
& \psi_d: S_d/I_{\mathbb{X}}(d) \rightarrow R_d/I_{\mathbb{T}_{k-1}}(d), \\
& \hspace{0.1cm} f+I_{\mathbb{X}}(d) \rightarrow \theta(f)+I_{\mathbb{T}_{k-1}}(d).
\end{align*}

$\psi_d$ is a surjective linear transformation and then $(S_d/I_{\mathbb{X}}(d))/\ker \psi_d$ is isomorphic, as a linear space, to $R_d/I_{\mathbb{T}_{k-1}}(d)$. But $\ker \psi_d=I_{\mathbb{Y}}(d)/I_{\mathbb{X}}(d)$, and thus
$$
H_{\mathbb{X}}(d)-\dim_K I_{\mathbb{Y}}(d)/I_{\mathbb{X}}(d)=H_\psi(d),
$$
and the result follows.
\end{proof}

Although in Section \ref{bounds} we gave some bounds for the regularity index of $S/I_{\mathbb{X}}$, where $\mathbb{X}$ is the toric set parameterized by the edges of a bipartite graph $\mathcal{G}$ with a perfect matching, in the following result we give a formula of this number in terms of the corresponding regularity index associated with the projective torus, which is given by (\cite[Lemma 1]{GRH})
\begin{equation} \label{regtorus}
{\rm{reg}} \, (R/I_{\mathbb{T}_{k-1}})=(q-2)(k-1),
\end{equation} 
and the regularity of $I_{\mathbb{Y}}/I_{\mathbb{X}}$.

\begin{corollary} \label{corollary}
The regularity index of the quotient ring $S/I_{\mathbb{X}}$ is given by
$$
{\rm{reg}} \,(S/I_{\mathbb{X}})=\max \{{\rm{reg}} \, (R/I_{\mathbb{T}_{k-1}}),{\rm{reg}} \, (I_{\mathbb{Y}}/I_{\mathbb{X}})\}.
$$
\end{corollary}

\begin{proof}
Let $\varphi$ the linear map
\begin{align*}
& \varphi: I_{\mathbb{Y}}(d)/I_{\mathbb{X}}(d) \rightarrow I_{\mathbb{Y}}(d+1)/I_{\mathbb{X}}(d+1), \\
& \hspace{0.8cm} f+I_{\mathbb{X}}(d) \rightarrow X_1f+I_{\mathbb{X}}(d+1).
\end{align*}

If $f+I_{\mathbb{X}}(d)=g+I_{\mathbb{X}}(d)$, with $f,g \in I_{\mathbb{Y}}(d)$, then
$$
X_1f-X_1g=X_1(f-g) \in I_{\mathbb{X}}(d+1),
$$
and thus $X_1f+I_{\mathbb{X}}(d+1)=X_1 g+I_{\mathbb{X}}(d+1)$.
It implies that $\varphi$ is a well--defined map. It suffices to show that this is an injective map because, in this case,  $H_{\psi}(d) \leq H_{\psi}(d+1)$ and $H_{\psi}$ is a non--decreasing function. Let $f +I_{\mathbb{X}}(d) \in \ker \varphi$. Then $X_1f \in I_{\mathbb{X}}(d+1)$. Let $$P=[t_1t_2,\ldots,t_1t_{2i_1},t_3t_4,\ldots,t_3t_{2i_3},\ldots,t_{2k-1}t_{2k},\ldots,t_{2k-1}t_{2i_{2k-1}}] \in \mathbb{X}.$$ Therefore, $t_1t_2f(P)=0$ for all $t_1,t_2 \in K^*$. That is, $f(P)=0$ for all $P \in \mathbb{X}$. Then $f \in I_{\mathbb{X}}(d)$ and it implies that
$$
\ker \varphi=I_{\mathbb{X}}(d).
$$

Thus $\varphi$ is an injective map, and the claim follows.

\end{proof}

\begin{remark} \label{remark}
From Corollary \ref{corollary} and Equation (\ref{regtorus}), we get that
$$
{\rm{reg}} \, (S/I_{\mathbb{X}}) \geq {\rm{reg}} \, (R/I_{\mathbb{T}_{k-1}})=(q-2)(k-1),
$$
and this is a tight lower bound because it is attained in the case of even cycles (\cite[Theorem 6.2]{vaz}), and when the graph $\mathcal{G}$ is a complete bipartite graph of the form $\mathcal{K}_{m,m}$ (\cite[Corollary 5.4]{sarabia1}). Both graphs are bipartite graphs with a perfect matching.
\end{remark}

\begin{corollary} \label{corollary2}
Let $\mathcal{G}$ be a graph such that each of its $m$ connected components is an even cycle $C_{2l_i}$ and let $\mathbb{X}$ be the toric set parameterized by its edges. Then
$$
{\rm{reg}} \, (S/I_{\mathbb{X}})=(q-2)(n-1).
$$

\begin{proof}
We notice that in this case $\sum_{i=1}^m \l_i=k$, where $n=s=2k$. Furthermore, by \cite[Theorem 6.3]{vaz}, we get that
$$
{\rm{reg}} \, (S/I_{\mathbb{X}}) \leq (q-2)(s-\sum_{i=1}^m l_i-1)=(q-2)(k-1).
$$

As $\mathcal{G}$ is a bipartite graph with a perfect matching, the claim follows by Remark \ref{remark}.
\end{proof}
\end{corollary}

\section{Example} \label{example}

Let $\mathcal{G}$ be the graph with two connected components, each of them a square (a cycle with four vertices), and let $\mathbb{X}$ be the toric set parameterized by its edges.

\bigskip
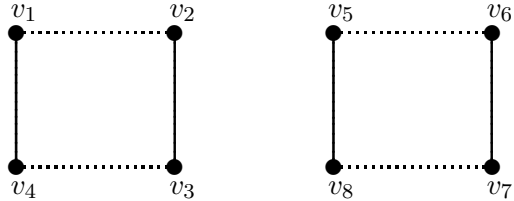
\begin{figure}[!ht]
\vspace{0.3cm}

\thicklines
\begin{picture}(800,50)
 \put(140,51){\circle*{6}}
 \put(200,51){\circle*{6}}
 \put(140,0){\circle*{6}}
 \put(200,0){\circle*{6}}
 
\put(138,57){$v_1$}
\put(198,57){$v_2$}
\put(198,-10){$v_3$}
\put(138,-10){$v_4$}

\curve(140,51,140,0)
\curve[20](140,51,200,51)
\curve(200,51,200,0)
\curve[20](140,0,200,0)

\vspace{0.3cm}
 \put(260,51){\circle*{6}}
 \put(320,51){\circle*{6}}
 \put(260,0){\circle*{6}}
 \put(320,0){\circle*{6}}
 
\put(258,57){$v_5$}
\put(318,57){$v_6$}
\put(318,-10){$v_7$}
\put(258,-10){$v_8$}

\curve(260,51,260,0)
\curve[20](260,51,320,51)
\curve(320,51,320,0)
\curve[20](260,0,320,0)

\end{picture}
\caption{The graph $\mathcal{G}$.}
\label{figure1}
 \end{figure}
 
A bipartition of the vertex set is given by $$\mathcal{U}=\{v_1,v_3,v_5,v_7\}, \,\,{\mbox{and}} \,\, \mathcal{W}=\{v_2,v_4,v_6,v_8\}.$$ 
 
 Also, a perfect matching is $\mathcal{M}=\{\{v_1,v_2\},\{v_3,v_4\},\{v_5,v_6\},\{v_7,v_8\}\}$, and its edges appear with dotted lines in Figure \ref{figure1}. The toric set parameterized by the edges of $\mathcal{G}$ is given by
 $$
 \mathbb{X}=\{[t_1t_2,t_1t_4,t_3t_4,t_3t_2,t_5t_6,t_5t_8,t_7t_8,t_7t_6] \in \mathbb{P}^7 : t_i \in K^*\},
 $$
 and the length of the code $C_{\mathbb{X}}(d)$ is $|\mathbb{X}|=(q-1)^{5}$. Moreover, by using Corollary \ref{corollary2},
 $$
 {\rm{reg}} \, (S/I_{\mathbb{X}})=3q-6, \,\,q>2.
 $$
 
 On the other hand, the set $\mathbb{Y}$, defined in Equation (\ref{Intersection}), is given by
 $$
 \mathbb{Y}=\{[t_1,t_1,t_3,t_3,t_5,t_5,t_7,t_7] \in \mathbb{P}^7 : t_i \in K^*\},
 $$ 
 and its vanishing ideal, according with Theorem \ref{theorem1}, is
 \begin{align*}
 I_{\mathbb{Y}}=(X_3^{q-1}-X_1^{q-1},X_5^{q-1}-X_1^{q-1},X_7^{q-1}- & X_1^{q-1}, X_2-X_1, \\
 & X_4-X_3,X_6-X_5,X_8-X_7).
 \end{align*}
 
 Furthermore, if we take $q=5$, then the Hilbert functions involved in Theo\-rem \ref{dimension} are described in Table \ref{tab1}. Of course, we are interested in the cases $1 \leq d<{\rm{reg}} \, (S/I_{\mathbb{X}})=9$.
 
 \begin{table}[h] 
\caption{The di\-ffe\-rent Hilbert functions involved in Theorem \ref{dimension} with $q=5$, and $\mathbb{X}$ being the toric set parameterized by the edges of the graph $\mathcal{G}$ of Figure \ref{figure1}.}
\label{tab1}
\begin{center}
\begin{tabular}{||c||c||c||c||c||c||c||c||c||} \hline
$d$ & 1 & 2 & 3 & 4 & 5 & 6 & 7 & 8 \\ \hline
$H_{\mathbb{T}_{k-1}}(d)$ & 4 & 10 & 20 & 32 & 44 & 54 & 60 & 63  \\ \hline
$H_{\psi}(d)$ & 4 & 24 & 84 & 208 & 396 & 616 & 796 & 912 \\ \hline
$H_{\mathbb{X}}(d)$ & 8 & 34 & 104 & 240 & 440 & 670 & 856 & 975  \\ \hline
\end{tabular}
\end{center}
\end{table}
 
Moreover, if we set $l_d=\left\lceil \frac{\delta_{\mathbb{T}_{n-1}}(2d)}{(q-1)^m}\right\rceil=\left\lceil   \frac{\delta_{\mathbb{T}_{7}}(2d)}{4^2} \right\rceil$, $u_d=(q-1)^{k-m} \cdot \delta_{\mathbb{T}_{k-1}}(d)=4^2 \cdot \delta_{\mathbb{T}_{3}}(d)$ (both bounds appear in Theorem \ref{theorem1}), and $B_d=(q-1)^{n-m-1}-H_{\mathbb{X}}(d)+1=4^5-H_{\mathbb{X}}(d)+1$ is the Singleton bound, then we get Table \ref{tab2}.

 \begin{table}[h]
\caption{Some bounds for the minimum distance of the code $C_{\mathbb{X}}(d)$ pa\-ra\-me\-te\-ri\-zed by the edges of the graph $\mathcal{G}$ of  Figure \ref{figure1}.}
\label{tab2}
\begin{center}
\begin{tabular}{||c||c||c||c||c||c||c||c||c||} \hline
$d$ & 1 & 2 & 3 & 4 & 5 & 6 & 7 & 8 \\ \hline
$l_d$ & 512 & 192 & 64 & 32 & 12 & 4 & 2 & 1   \\ \hline
$u_d$ & 768 & 512 & 256 & 192 & 128 & 64 & 48 & 32 \\ \hline
$B_d$ & 1017 & 991 & 921 & 785 & 585 & 355 & 169 & 50  \\ \hline
\end{tabular}
\end{center}
\end{table}

\end{document}